\documentclass[twoside, 12pt]{article}

\usepackage{amscd} 
 
\usepackage{amsmath,stackrel}
\usepackage{amsthm}
\usepackage{amstext}
\usepackage{amsfonts}
\usepackage{latexsym}

\newcommand{\Ad}{\ensuremath{\mathrm{Ad}}}
\newcommand{\g}{\ensuremath{\mathfrak g}}
\newcommand{\gl}{\ensuremath{\mathfrak gl}}
\newcommand{\go}{\ensuremath{{\mathfrak g}_0}}

\newcommand{\ind}{\ensuremath{\mathrm{ind}}}
\newcommand{\kk}{\ensuremath{\mathfrak k}}
\newcommand{\kko}{\ensuremath{{\mathfrak k}_0}}
\newcommand{\p}{\ensuremath{\mathfrak p}}
\newcommand{\po}{\ensuremath{{\mathfrak p}_0}}
\newcommand{\so}{\ensuremath{\mathfrak so}}
\newcommand{\U}{\ensuremath{\mathfrak U}}

\newcommand{\C}{\ensuremath{\mathbb C}}
\newcommand{\R}{\ensuremath{\mathbb R}}
\newcommand{\N}{\ensuremath{\mathbb N}}
\newcommand{\Z}{\ensuremath{\mathbb Z}}

\newcommand{\vr}{\vspace{2mm}}

\newtheorem{definition}{Definition}
\newtheorem{theorem}{Theorem}
\newtheorem{proposition}{Proposition}
\newtheorem{lemma}{Lemma}

\newcommand{\ds}{\displaystyle}
 
\begin{document}

\begin{center}
  {\bf  \Large On unitary representations of disconnected real reductive groups}
\end{center}
\vspace{5mm}

\begin{center}

D. Kova\v{c}evi\'{c}{\footnote{e-mail:domagoj.kovacevic@fer.hr}}\\
  University of Zagreb, Faculty of Electrical Engineering and Computing,\\
  Unska 3, HR-10000 Zagreb, Croatia

\end{center}
\setcounter{page}{1}
\vspace{7mm}

\begin{abstract}
  Let $G$ be the real reductive group and let $G_0$ be the identity component.
  Let us assume that the unitary dual $\widehat{G_0}$ is known. In this
  paper (in Section \ref{ind}) the unitary dual $\widehat{G}$ is constructed.
  Automorphisms of $G_0$ generated by elements of $G$ are the main
  ingredient of the construction. If the automorphism is outer, one has to
  consider the corresponding intertwining operators $S$.
  Operators $S$ and their properties are analyzed in Section \ref{s}.
  Automorphisms of \go\ are closely related to automorphisms of $G_0$.
  They are investigated in Section \ref{aut}.
  Automorphisms of \so(4,4) are analyzed in Subsection \ref{so44}.
\end{abstract}

\section{Introduction}

Let $G$ be a real reductive group. The definition of the real reductive group
will be introduced in the Section \ref{not}. Our definition is more general
then the usual one. Namely, we do not require that Ad($x$) is inner
automorphism of \g\ for all $x\in G$.

The main goal of this paper is to describe the unitary dual of the group $G$
in terms of the unitary dual of the identity component $G_0$ of $G$.
The natural choice for this construction is the induction from  $G_0$ to $G$.
Since $G_0$ is a normal subgroup of $G$, every element of $G$ produces
an automorphism of $G_0$. These automorphisms also play important role
in our construction.

In Section \ref{aut}, we analyze automorphisms of \go.
Namely, there is isomorphism from Aut $G_0$ onto the subgroup of Aut \go,
the group of the automorphisms of \go. This subgroup is equal to Aut \go,
if $G_0$ is simply connected. For more details, see \cite{bo}. Hence, the
group Aut~\go, is close to the group Aut~$G_0$. The most complicated situation
(for simple Lie algebras \go) appears if \go\ is equal to \so(4,4)\ when Aut
\go\ is equivalent to $S_4$ and it is analyzed in Subsection \ref{so44}.
These results are already known but it is hard to find a reference.
In particular, we believe that this approach can not be found elsewhere.

Let $(\pi,V)$ be the representation of $G_0$ and $x$ representative of some
component of $G$. Then we consider the representation $\pi^x$ defined by
$\pi^x(g)=\pi(x^{-1}gx)$. If $\pi^x$ is equivalent to $\pi$ then the
corresponding intertwining operator is denoted by $S_x$. The subgroup
of $G$ containing all such elements $x$ is denoted by $G_2$. Operators
$S_x$ naturally lead to coefficients $\beta(x,y)$ which give a better
understanding of the structure of $G_2$.
Operators $S_x$, coefficients $\beta(x,y)$ and some important properties are
described in Section \ref{s}.

In Section \ref{ind} we analyze the induction procedure that goes from
$G_0$ to various subgroups of $G$. Actually, there is a sequence of
subgroups $G_0\leq G_1\leq G_2\leq G$ where $G_1$ is a subgroup of $G$
such that conjugation by any $x\in G_1$ is equal to conjugation by some
$g\in G_0$. In Subsection \ref{ind01}, we analyze the induction step from
$G_0$ to $G_1$ and get reducibility in terms of some finite subgroups.
This result is mentioned in \cite{sc} for discrete series and  the request on
the Lie group is more restrictive (see also the discussion in the Section 9 of
\cite{sc}).
The next step is to analyze induction that goes from $G_1$ to $G_2$.
However, it is easier to consider the induction from $G_0$
to $G_2$. This time reducibility is described in terms of abstract finite
groups and it is derived in Subsection \ref{ind02}. For this induction
step we need intertwining operators $S_x$ and some properties of coefficients
$\beta(x,y)$. One could say that Theorem~\ref{gogi} is a special case of Theorem
\ref{gogii}. However, the statement of Theorem \ref{gogi} is more natural.
Finally, in Subsection \ref{ind03}, is shown that
$\ind_{G_2}^G\pi$ is irreducible if $\pi$ is irreducible. It is well known
result (see \cite{du}), but we add it in order to complete our picture.
Our induction procedure is written in a different way (for example,
see \cite{fh}), but this terminology is more convenient for calculations.
We also add several examples in order to demonstrate our results.
These results are also well known, but they describe our theorems and
proofs.

The main idea of this paper is to use representations of the group which is
slightly bigger (but still finite) then a quotient group $G/G_0$
in order to describe induced representation $\ind_{G_0}^{G_2}V$
in terms of finite groups.
Similar ideas can be found in \cite{ma}. Coefficients $\beta$
(Definition \ref{beta}) and the subgroup $F_2$ (Subsection \ref{ind02})
correspond to the group extension mentioned in \cite{ma}.
Representation $\eta_m$ of $F_0$ corresponds to projective representations
from \cite{ma}.
However, this paper is oriented to unitary representations of the disconnect
group $G$ and the description is very precise (Theorems \ref{gogi} and
\ref{gogii}). We also analyze coefficients $\beta$ and show that
$\beta^n=1$ where $n=|G/G_0|$ (Proposition \ref{betaroot}).
Finally our constructions are very explicit and can be easily performed for
small $n$.
Paper \cite{ma} is much more general and can not be used to derive (easily)
results of this paper. It is oriented to "Mackey theory".
In order to understand the flavor of that paper, it is enough to analyze
theorems of Section 8 and examples in Section 9.

The obtained results can be also used to analyze multiplicities of irreducible
components obtained by parabolic induction
(for a different approach, see \cite{ga}).
There is an example of group $G$ for which multiplicities are 2 (see \cite{ga}
or \cite{sv}).
The identity component $G_0$ has the form
$(SL(2,\R)\times SL(2,\R))/{\pm(I\times I)}$ and $G/G_0\cong\Z_2\times Z_2$.
This example is analyzed in Subsection \ref{ind02}.
Let us denote by $\rho$ the irreducible representation of $P_0=P\cap G_0$
such that $\ind_{P_0}^{G_0}\rho=\pi_1\oplus\pi_2\oplus\pi_3\oplus\pi_4$.
Let us denote by $x$ and $y$ elements of the form $x=\mathrm{diag}(1,-1,1,1)$
and $y=\mathrm{diag}(1,1,1,-1)$. Let us multiply $x$ by $i$ and $y$ by $j$
in order to obtain $xy=-yx$. Now, $\rho^x\cong\rho$ and $\rho^y\cong\rho$,
but $\left(\ind_{P_0}^{P_0\cup xP_0}\rho\right)^y\not\cong
\ind_{P_0}^{P_0\cup xP_0}\rho$. Hence, $\ind_{P_0}^P\rho=\tau\oplus\tau$.
Also, we could use Theorem \ref{gogi} and use that
$\ind_{\Z_2}^Q\eta=\zeta\oplus\zeta$
where the quaternion group $Q=Z(G)$, $\Z_2=Q\cap G_0$ and $\eta$ is a
nontrivial representation.
Since, $\pi_1^x\cong\pi_2$, $\pi_1^y\cong\pi_3$ and $\pi_1^{xy}\cong\pi_4$,
$\ind_{G_0}^G\pi_l=\sigma$ for $l\in\{1,2,3,4\}$.
Now, the induction $\ind_{P_0}^G$ can be described by the following diagram
\begin{equation}\nonumber
  \begin{CD}
    \rho @>\ind_{P_0}^{G_0}>> \pi_1\oplus\pi_2\oplus\pi_3\oplus\pi_4\\
    @VV\ind_{P_0}^PV @VV\ind_{G_0}^GV\\
    \tau\oplus\tau @>\ind_P^G>>\sigma\oplus\sigma\oplus\sigma\oplus\sigma.
  \end{CD}
\end{equation}
It follows that $\ind_P^G\tau=\sigma\oplus\sigma$. The terminology is
explained in Section \ref{s}.

\section{Notation}\label{not}

For us, the real reductive group $G$ will be the Lie group $G$ with a
compact subgroup $K$, an involution $\theta$ and a nondegenerate,
Ad(G) invariant, $\theta$ invariant bilinear form $B$ such that 
\begin{enumerate}
  \item the corresponding Lie algebra \go\ is a reductive Lie algebra,
  \item $\go=\kko\oplus\po$, \kko\ is the Lie algebra of $K$,
  \item \kko\ and \po\ are orthogonal under $B$ and $B$ is positive
    definite on \po\ and negative definite on \kko,
  \item $G=K\times\exp\po$.
\end{enumerate}

The definition is taken from \cite{kn}, but we do not take the last
condition: every automorphism $\Ad(g)$ of $\g=\go\otimes_\R\C$ is inner for \g.
Let $G_0$ and $K_0$ be identity components of $G$ and $K$ respectively.
The quotient $G/G_0\cong K/K_0$ is finite since $K$ is compact.

The representation of $G$ will be denoted by $\pi$ or $(\pi,V)$.
We say that $V$ is a $G$-module and sometimes write $g.v$ or simply $gv$
instead of $\pi(g)v$.
A representation of \g\ on $V$ will be denoted by $\pi$ again.
It naturally extends to a homomorphism $\pi:\U(\g)\rightarrow End(V)$.
We say that $V$ is a \g-module and sometimes write $X.v$ instead of $\pi(X)v$.

Let $(\pi,V)$ be a representation of a real Lie group $G$.
A $(\g,K)$-module $V$ is a \g-module and $K$-module such that the following
conditions are satisfied:
\begin{enumerate}
  \item $k.(X.v)=(Ad(k)X).(k.v)$ for $v\in V_F$ where $V_F$ is the set of
    $K$-finite differentiable vectors in $V$. For more details, see \cite{ba}.
  \item The space $\{k.v\,|\,k\in K\}$ is finite-dimensional for any $v\in V$
    and the action is $C^\infty$.
  \item $\ds \left.\frac{\mathrm{d}}{\mathrm{d}t}\right|_{t=0}
    (\mathrm{exp}\,tX).v=X.v$ for $v\in V_F$.
\end{enumerate}
The first condition can be written in the form
\begin{equation}\label{b4}
  X.(k.v)=k.(\Ad(k^{-1})X.v).
\end{equation}
To each irreducible unitary representation there corresponds $(\g,K)$-module
and vice verse. Hence, the study of irreducible unitary representations can
be done via $(\g,K)$-modules. More details can be found in \cite{ba}.

\section{Automorphisms of \go}\label{aut}

We are interested in automorphisms of the Lie group $G$.
These automorphism form a subgroup of automorphisms of the corresponding
Lie algebra \go. If $G_0$ is simply connected then the group of automorphisms
of $G_0$ is equal to the group of automorphisms of \go\ (see \cite{bo}).
Hence, we will consider a simpler problem: finding automorphisms of \go.

For the rest of this section we assume that \go\ is simple.
Let $\go=\kko\oplus\po$ be the Cartan decomposition and we fix \kko\ and \po.
Complexification of \kko\ will be denoted by \kk\ and complexification 
of \po\ by \p.

If \go\ is compact then $\mathrm{Aut}_\R\go/\mathrm{Int}\go$ is isomorphic
to the group of automorphisms of the Dynkin diagram of \g. The same is
true for $\mathrm{Aut}_\C\g/\mathrm{Int}\g$. See, for example \cite{kn}.
For the general \go, this group is slightly more complicated.

\begin{theorem}\label{out}
  Outer automorphism $\varphi_0$ of simple Lie algebra \go\ can be obtained as
  follows:
  \begin{enumerate}
    \item [(a)] if \kko\ contains $\R$ as a summand then
      $\varphi_0|_{\R}\equiv-1$
    \item [(b)] if \kko\ contains isomorphic components then $\varphi_0$ mixes
      these components
    \item [(c)] if $\varphi_0|_{\kko}\equiv1$ then $\varphi$ is equal to the Cartan
      involution.
  \end{enumerate}
\end{theorem}

\begin{proof}
We can assume that \kko\ and \po\ are fixed under the action of $\varphi_0$.
Hence, $\varphi_0|_{\kko}$ is the automorphism of \kko\ and
$\varphi_0|_{\po}$ is the automorphism of \po.
The complexification of $\varphi_0$ will be denoted by $\varphi$.

If \kko\ contains \R, then $\kko=\kko^1\oplus\R$ for some semisimple algebra
$\kko^1$. In that situation, $G/K$ is Hermitian, there is multiplication-by-$i$
map $J$ (acting on \po\ and \p)
with the property $J^2=-1$ and $\p=\p^+\oplus\p^-$ where
$\p^+$ is the $+i$ eigenspace for $J$ and $p^-$ is the $-i$ eigenspace for $J$.
Also, \R\ is the center of \kko, $\varphi_0$ maps \R\ to \R\ and
there is $X_0\in\R$ such that $[X_0,Y]=JY$ for $Y\in\p$.
(For more details, see \cite{kn}.)
Let us denote the complexification of the center \R\ by \C.
Then $\varphi(X_0)=\alpha X_0$ for $\alpha\in\C$.
Let us write $\varphi(Y)=Z+T$ for $Y,Z\in\p^+$ and $T\in\p^-$. Now,
\begin{equation}\nonumber
  [X_0,Y]=iY.
\end{equation}
The action of $\varphi$ produces
\begin{equation}\nonumber
  [\alpha X_0,Z+T]=i\left(Z+T\right).
\end{equation}
Since the left hand side is equal to $\alpha[X_0,Z+T]=i\alpha Z-i\alpha T$,
either $\alpha=1$ and $T=0$ ($ad(X)$ maps $\p^+$ to $\p^+$ and $\p^-$ to $\p^-$)
or $\alpha=-1$ and $Z=0$ ($ad(X)$ maps $\p^+$ to $\p^-$ and $\p^-$ to $\p^+$).
It shows that $\varphi(X_0)=\pm X_0$.

Now, let us consider the situation when \kko\ does not contains \R\ as a
component. If \kko\ contains isomorphic components, then outer isomorphism can
mix these components. A nice example is given in Subsection \ref{so44}.

It remains to consider the case when \kko\ does not contain neither \R\ nor
isomorphic components. Let us consider an outer automorphisms $\varphi_0$
such that $\varphi|_{\kko}\equiv1$.
Since $G/K$ is not Hermitian, the representation of \kk\ on \p\ is irreducible.
Schur's lemma shows that $\varphi$ is a multiplication by $\lambda\in\C$.
Since $\varphi$ is a complexification of $\varphi_0$, $\lambda\in\R$.
Finally, take $X,Y\in\po$ such that $[X,Y]\neq0$. It shows that $\lambda^2=1$.
\end{proof}

If $G=SL(2,\R)$, then the outer automorphism is obtained as conjugation by
$\mathrm{diag}(1,-1)$. Now, $\kko=\R$ and $\varphi_0|_{\R}\equiv-1$.
If $G=SL(n,\R)$, $n>2$, the outer automorphism is the Cartan involution.

\subsection{Automorphisms of \so(4,4)}\label{so44}

Let us recall that
\begin{equation}\label{c4}
  \so(4,4)=\left\{X\in\gl(8,\R)\,|\,X^*I_{4,4}+I_{4,4}X=0\right\}=
  \left[\begin{array}{cc}A&B\\B^*&C\end{array}\right]
\end{equation}
where $A$, $B$ and $C$ are 4-by-4 real matrices and $A$ and $C$ are
skew-symmetric.
Then $\kko=\so(3)\times\so(3)\times\so(3)\times\so(3)$.
Let us write $\kko=\prod_{i=1}^4\so(3)_i$. Generators of $\so(3)_1$ and
$\so(3)_2$ are in $A$ (see (\ref{c4})) and given by
\begin{equation}\nonumber
  X_1=\left[\begin{array}{cccc}
    0&1&0&0\\-1&0&0&0\\0&0&0&-1\\0&0&1&0
  \end{array}\right]\hfill
  Y_1=\left[\begin{array}{cccc}
      0&0&1&0\\0&0&0&1\\-1&0&0&0\\0&-1&0&0
    \end{array}\right]\hfill
  Z_1=\left[\begin{array}{cccc}
      0&0&0&1\\0&0&-1&0\\0&1&0&0\\-1&0&0&0
    \end{array}\right]
\end{equation}
and
\begin{equation}\nonumber
  X_2=\left[\begin{array}{cccc}
    0&1&0&0\\-1&0&0&0\\0&0&0&1\\0&0&-1&0
  \end{array}\right]\hfill
  Y_2=\left[\begin{array}{cccc}
      0&0&-1&0\\0&0&0&1\\1&0&0&0\\0&-1&0&0
    \end{array}\right]\hfill
  Z_2=\left[\begin{array}{cccc}
      0&0&0&1\\0&0&1&0\\0&-1&0&0\\-1&0&0&0
    \end{array}\right]
\end{equation}
Generators of $\so(3)_3$ and $\so(3)_4$ have the same form as generators of
$\so(3)_1$ and $\so(3)_2$ respectively, but they are in $C$.
Elements of \po\ are in $B$
\begin{theorem}
  The group of outer automorphisms of \so(4,4) is $S_4$.
\end{theorem}
\begin{proof}
The group $S_4$ permutes components of \kko. We have to show that these
automorphisms of \kko\ can be extended to \go. It is enough since \R\ is not a
summand of \kko\ and the conjugation by $\mathrm{diag}(1,1,1,1,-1,-1,-1,-1)$
produces the Cartan involution.

It is easy to produce an automorphism of \go\ that sends $\so(3)_1$ to
$\so(3)_2$, $\so(3)_2$ to $\so(3)_1$ and fixes $\so(3)_3$ and $\so(3)_4$:
$\varphi_{12}(X)=x_{12}Xx_{12}^{-1}$ where
$x_{12}=\mathrm{diag}(1,1,-1,1,1,1,1,1)$. Similarly,
$x_{34}=\mathrm{diag}(1,1,1,1,1,1,-1,1)$ produces an isomorphism $\varphi_{34}$.
The conjugation by
\begin{equation}\nonumber
  \left[\begin{array}{cc}0&I_4\\I_4&0\end{array}\right]
\end{equation}
produces an automorphism that permutes $\so(3)_1$ and $\so(3)_3$ as well as
$\so(3)_2$ and $\so(3)_4$.

It remains to construct the the isomorphism $\varphi_{1234}=\varphi$ which
sends $X_i$, $Y_i$ and $Z_i$ to $X_{i+1}$, $Y_{i+1}$ and $Z_{i+1}$ respectively.
It is enough to write the action of $\varphi$ on components in $B$. So,
$\varphi$ sends
\begin{scriptsize}
  \begin{equation}\nonumber
    \left[\begin{array}{cccc}
      1&0&0&0\\0&1&0&0\\0&0&1&0\\0&0&0&1
    \end{array}\right]\hfill
    \left[\begin{array}{cccc}
      0&1&0&0\\-1&0&0&0\\0&0&0&-1\\0&0&1&0
    \end{array}\right]\hfill
    \left[\begin{array}{cccc}
      0&0&1&0\\0&0&0&1\\-1&0&0&0\\0&-1&0&0
    \end{array}\right]\hfill
    \left[\begin{array}{cccc}
      0&0&0&1\\0&0&-1&0\\0&1&0&0\\-1&0&0&0
    \end{array}\right]\hfill
  \end{equation}
  \begin{equation}\nonumber
    \left[\begin{array}{cccc}
      1&0&0&0\\0&1&0&0\\0&0&-1&0\\0&0&0&-1
    \end{array}\right]\hfill
    \left[\begin{array}{cccc}
      0&1&0&0\\-1&0&0&0\\0&0&0&1\\0&0&-1&0
    \end{array}\right]\hfill
    \left[\begin{array}{cccc}
      0&0&-1&0\\0&0&0&1\\1&0&0&0\\0&-1&0&0
    \end{array}\right]\hfill
    \left[\begin{array}{cccc}
      0&0&0&1\\0&0&1&0\\0&-1&0&0\\-1&0&0&0
    \end{array}\right]\hfill
  \end{equation}
  \begin{equation}\nonumber
    \left[\begin{array}{cccc}
      1&0&0&0\\0&-1&0&0\\0&0&1&0\\0&0&0&-1
    \end{array}\right]\hfill
    \left[\begin{array}{cccc}
      0&1&0&0\\1&0&0&0\\0&0&0&1\\0&0&1&0
    \end{array}\right]\hfill
    \left[\begin{array}{cccc}
      0&0&1&0\\0&0&0&-1\\1&0&0&0\\0&-1&0&0
    \end{array}\right]\hfill
    \left[\begin{array}{cccc}
      0&0&0&1\\0&0&1&0\\0&1&0&0\\1&0&0&0
    \end{array}\right]\hfill
  \end{equation}
  \begin{equation}\nonumber
    \left[\begin{array}{cccc}
      1&0&0&0\\0&-1&0&0\\0&0&-1&0\\0&0&0&1
    \end{array}\right]\hfill
    \left[\begin{array}{cccc}
      0&1&0&0\\1&0&0&0\\0&0&0&-1\\0&0&-1&0
    \end{array}\right]\hfill
    \left[\begin{array}{cccc}
      0&0&1&0\\0&0&0&1\\1&0&0&0\\0&1&0&0
    \end{array}\right]\hfill
    \left[\begin{array}{cccc}
      0&0&0&1\\0&0&-1&0\\0&-1&0&0\\1&0&0&0
    \end{array}\right]\hfill
  \end{equation}
\end{scriptsize}
to
\begin{scriptsize}
  \begin{equation}\nonumber
    \left[\begin{array}{cccc}
      -1&0&0&0\\0&-1&0&0\\0&0&-1&0\\0&0&0&-1
    \end{array}\right]\hfill
    \left[\begin{array}{cccc}
      0&-1&0&0\\1&0&0&0\\0&0&0&-1\\0&0&1&0
    \end{array}\right]\hfill
    \left[\begin{array}{cccc}
      0&0&1&0\\0&0&0&-1\\-1&0&0&0\\0&1&0&0
    \end{array}\right]\hfill
    \left[\begin{array}{cccc}
      0&0&0&-1\\0&0&-1&0\\0&1&0&0\\1&0&0&0
    \end{array}\right]\hfill
  \end{equation}
  \begin{equation}\nonumber
    \left[\begin{array}{cccc}
      1&0&0&0\\0&1&0&0\\0&0&-1&0\\0&0&0&-1
    \end{array}\right]\hfill
    \left[\begin{array}{cccc}
      0&1&0&0\\-1&0&0&0\\0&0&0&-1\\0&0&1&0
    \end{array}\right]\hfill
    \left[\begin{array}{cccc}
      0&0&1&0\\0&0&0&1\\-1&0&0&0\\0&-1&0&0
    \end{array}\right]\hfill
    \left[\begin{array}{cccc}
      0&0&0&1\\0&0&-1&0\\0&1&0&0\\-1&0&0&0
    \end{array}\right]\hfill
  \end{equation}
  \begin{equation}\nonumber
    \left[\begin{array}{cccc}
      1&0&0&0\\0&-1&0&0\\0&0&1&0\\0&0&0&-1
    \end{array}\right]\hfill
    \left[\begin{array}{cccc}
      0&1&0&0\\1&0&0&0\\0&0&0&-1\\0&0&-1&0
    \end{array}\right]\hfill
    \left[\begin{array}{cccc}
      0&0&-1&0\\0&0&0&-1\\-1&0&0&0\\0&-1&0&0
    \end{array}\right]\hfill
    \left[\begin{array}{cccc}
      0&0&0&1\\0&0&-1&0\\0&-1&0&0\\1&0&0&0
    \end{array}\right]\hfill
  \end{equation}
  \begin{equation}\nonumber
    \left[\begin{array}{cccc}
      1&0&0&0\\0&-1&0&0\\0&0&-1&0\\0&0&0&1
    \end{array}\right]\hfill
    \left[\begin{array}{cccc}
      0&1&0&0\\1&0&0&0\\0&0&0&1\\0&0&1&0
    \end{array}\right]\hfill
    \left[\begin{array}{cccc}
      0&0&-1&0\\0&0&0&1\\-1&0&0&0\\0&1&0&0
    \end{array}\right]\hfill
    \left[\begin{array}{cccc}
      0&0&0&1\\0&0&1&0\\0&1&0&0\\1&0&0&0
    \end{array}\right]\hfill
  \end{equation}
\end{scriptsize}
respectively.
\end{proof}

\section{Operators $S$}\label{s}

We analyze two subgroups of $G$: $G_1$ and $G_2$. We start with the
smaller one.

\begin{definition}\label{defg1}
  The subgroup $G_1$ of $G$ is defined by
  \begin{equation}\nonumber
    G_1=\{x\in G\,|\,x^{-1}gx=h^{-1}gh,\;\mbox{for some}\;
    h\in G_0,\;\forall g\in G_0\}.
  \end{equation}
  The center $Z_{G_0}$ of $G_0$ will be denoted by $Z_0$.
  The centralizer $C_G(G_0)$ of $G_0$ in $G_1$ will be denoted by $Z_1$
\end{definition}

It is important to notice that $G_1$ contains connected components which
intersect $C_G(G_0)$. Hence, $Z=C_G(G_0)=C_{G_1}(G_0)$.

\begin{proposition}\label{p1}
  The subgroup $G_1$ is a normal subgroup of $G$.
\end{proposition}
\begin{proof}
Let $y\in G_1$, $x\in G$ and $g\in G_0$. We have to show that $x^{-1}yx\in G_1$.
Then
\begin{gather}
  (x^{-1}yx)^{-1}g(x^{-1}yx)=x^{-1}y^{-1}xgx^{-1}yx=\nonumber\\
  x^{-1}h^{-1}xgx^{-1}hx=(x^{-1}hx)^{-1}g(x^{-1}hx).\nonumber
\end{gather}
Hence, the conjugation by $x^{-1}yx$ is the same as conjugation by $x^{-1}hx$.
\end{proof}

Let $\pi$ be the representation of $G_0$ on the vector space $V$.
Let $x\in G\setminus G_0$ be any representative of the class $xG_0$.
In order to understand the representation $\ind_{G_0}^G\pi$, one considers
representations $\pi^x(g)=\pi(x^{-1}gx)$ where $x\in G\setminus G_0$
and $g\in G_0$.
Let us assume that $\pi^x$ is equivalent to $\pi$.
Then there exists an intertwining operator $S_x$ such that
\begin{equation}\label{d4}
  \pi^x(g)S_x=S_x\pi(g).
\end{equation}
It should be more correct to write $S_x(\pi)$, but we write $S_x$ since
$\pi$ is fixed.

\vr
\noindent{\bfseries Remark.} Elements $x\in G_2\setminus G_1$ correspond to
outer automorphisms of the group $G_0$. Theorem \ref{out} shows that there are
3 types of such automorphisms. Operators $S_x$ can be easily described for the
first type of outer automorphisms. Let $G=SL(2,\R)$, $V$ be the $(\g,K)$-module
and $V=\bigoplus_{j\in\Z}V_j$ where $V_j$ are $K$-types.
Let $x=\left[\begin{array}{cc}1&0\\0&-1\end{array}\right]$.
Then the conjugation by $x$ is the outer automorphism and $S_xv_j=v_{-j}$
(see (\ref{e64})). More details will be provided later.
\vr

If $x,y\in G$ have the property that $\pi^x$ and $\pi^y$ are equivalent to $\pi$
then $\pi^{xy}$ is equivalent to $\pi$ and (one possible) intertwining
operator $S_{xy}$ is equal to
\begin{equation}\label{d6}
  S_{xy}=S_yS_x.
\end{equation}
Also, it is easy to see that one possible intertwining operator for
$x^{-1}$ is $S_x^{-1}$.

\begin{definition}
  The subgroup $G_2$ is the set of elements $x\in G$
  for which $\pi^x$ is equivalent to $\pi$.
\end{definition}

If the representation $\pi$ is irreducible, operator $S_x$ is unique up to a
scalar value. It follows from (\ref{d6}) that $S_x^2$ is the intertwining
operator for $\pi^{x^2}$ and in general
\begin{equation}\label{d8}
  \pi^{x^k}(g)S_x^k=S_x^k\pi(g),\;\forall k\in\N\,(\Z).
\end{equation}
Since $G/G_0$ is finite, $x^{n(x)}=h(x)\in G_0$ for some $n(x)\in\N$
and (\ref{d8}) transforms to
\begin{equation}\nonumber
  \pi((x^{n(x)})^{-1})\pi(g)\pi(x^{n(x)})S_x^{n(x)}=S_x^{n(x)}\pi(g).
\end{equation}
It shows that $\pi(h(x))S_x^n$ is the intertwining operator for the
representation $\pi$. Since $\pi$ is irreducible, Schur's Lemma shows that
$\pi(h(x))S_x^{n(x)}=\lambda I$ for some $\lambda\in\C$. Hence, we can
take $\ds \frac1{\sqrt[n(x)]{\lambda}}S_x$ instead of $S_x$ in order to obtain
\begin{equation}\label{d12}
  \pi(h(x))S_x^{n(x)}=I.
\end{equation}
Let $\xi_n=\cos\frac{2\pi}{n}+i\sin\frac{2\pi}{n}$.
Then request (\ref{d12}) puts another restriction on $S_x$: $S_x$ is unique
up to a power of $\xi_{n(x)}$.

If $h\in G_0$, then
\begin{equation}\label{d13}
  S_h=\pi(h^{-1}).
\end{equation}
Relation (\ref{d12}) is satisfied.
Also, (\ref{d6}) is satisfied.

If $x\in G_1$, say $x^{-1}gx=h^{-1}gh$ for $h\in G_0$,
then $\pi^x$ is equivalent to $\pi$ and $S_x=\pi(h^{-1})$.
It is easy to see that (\ref{d12}) is satisfied.
Now, let us assume that $x\in G_2/G_0$ and $S_x$ are fixed. For $h,g\in G_0$,
$S_{xh}$ and $S_{gx}$ are defined such that (\ref{d6}) and (\ref{d13})
satisfied, i.e.
\begin{equation}\nonumber
  S_{xh}=\pi(h^{-1})S_x\quad\mbox{and}\quad S_{gx}=S_x\pi(g^{-1}).
\end{equation}
We have to show that these definitions are good i.e. $S_{xh}=S_{gx}$ if
$xh=gx$. However, the first we show that operators $S_{xh}$ and $S_{gx}$
satisfy (\ref{d12}).

\begin{lemma}\label{Sgxh}
	Operators $S_{xh}$ and $S_{gx}$ satisfy (\ref{d12}).
\end{lemma}
\begin{proof}
Let us prove the lemma for operators $S_{xh}$.
\[ \pi((xh)^{n(x)}(\pi(h^{-1})S_x)^{n(x)}=
   \pi(xh\ldots x)S_x\pi(h^{-1})\ldots S_x= \]
\[ =\pi\left(x^{n(x)}x^{1-n(x)}hx^{n(x)-1}x^{2-n(x)}h\ldots x^{-1}hx\right)
   S_x\pi(h^{-1})S_x\pi(h^{-1})\ldots S_x= \]
\[ =\pi(x^{n(x)})\pi^{x^{n(x)-1}}(h)\pi^{x^{n(x)-2}}(h)\ldots\pi^{x^2}(h)
   \underbrace{\pi^x(h)S_x}\pi(h^{-1})S_x\pi(h^{-1})\ldots
   S_x\stackrel{(\ref{d4})}{=} \]
\[ =\pi(x^{n(x)})\pi^{x^{n(x)-1}}(h)\pi^{x^{n(x)-2}}(h)\ldots\pi^{x^2}(h)
   \underbrace{S_x\pi(h)}\pi(h^{-1})S_x\pi(h^{-1})\ldots S_x= \]
\[ =\pi(x^{n(x)})\pi^{x^{n(x)-1}}(h)\pi^{x^{n(x)-2}}(h)\ldots\pi^{x^2}(h)
   S_x^2\pi(h^{-1})S_x\pi(h^{-1})\ldots S_x\stackrel{(\ref{d8})}{=}
   \ldots\stackrel{(\ref{d8})}{=} \]
\[ \stackrel{(\ref{d8})}{=}\pi(x^{n(x)})\pi^{x^{n(x)-1}}(h)
   S_x^{n(x)-1}\pi(h^{-1})S_x\stackrel{(\ref{d8})}{=}
   \pi(x^{n(x)})S_x^{n(x)}\stackrel{(\ref{d12})}{=}I. \]
The proof for the operators $S_{gx}$ is similar.
\end{proof}

\begin{proposition}
	Let $x\in G_2$ and $h,g\in G_0$. If $xh=gx$ then
	\begin{equation}\nonumber
	  S_{xh}=S_{gx}.
	\end{equation}
\end{proposition}
\begin{proof}
Operators $S_{xh}$ and $S_{gx}$ are intertwining operators for $xh=gh$.
Hence, $S_{xh}S_{gx}^{-1}=\lambda I$ for some $\lambda\in\C$ and the
identity operator $I$ on $V$. Lemma \ref{Sgxh} shows that $\lambda=\xi_{n(x)}$.
For given $h\in G_0$ and $x\in G_2$, $g(h)=xhx^{-1}$ is a continuous function.
Now, let us consider a function
\begin{equation}\nonumber
  f(h)=S_{xh}S_{gx}^{-1}=\pi(h^{-1})S_x\pi(xhx^{-1})S_x^{-1}.
\end{equation}
Since $f$ is a continuous function and the codomain is a discrete set,
$f$ is a constant function. Since $f(e)=I$, $f(h)=I$ for all $h\in G_0$
and $S_{xh}=S_{gx}$.
\end{proof}

It shows that we can work with the class $xG_0$ instead with a particular
representative.

Let us assume that intertwining operators $S_x$ are chosen and fixed
for all elements (classes) of $G_2$.
Now, let us consider two particular elements $x,y\in G_2$.
The operator $S_yS_x$ is the intertwining operator for $xy$ but it may not
satisfy (\ref{d12}). Even if it satisfies (\ref{d12}), it can happen that
$S_yS_x=\xi_{n(xy)}^kS_{xy}$ for some $k$.
However, by Schur's lemma, we know that $S_yS_x$ and $S_{xy}$ are proportional.
Hence, we define coefficients $\beta$.

\begin{definition}\label{beta}
  For the given choice of $S_x$, we define
  \begin{equation}\label{d16}
    S_yS_x=\beta(x,y)S_{xy}.
  \end{equation}
  for some $\beta(x,y)\in\C$.
\end{definition}

Coefficients $\beta(x,y)$ encode information related to reducibility of
$\ind_{G_0}^{G_2}\pi$. Hence, let us analyze properties of them.
We will need it in Subsection \ref{ind02}.

\begin{lemma}\label{as}
  Let $x,y,z\in G_2/G_0$. Then
  \begin{equation}\nonumber
    \beta(x,y)\beta(xy,z)=\beta(x,yz)\beta(y,z).
  \end{equation}
\end{lemma}
\begin{proof}
Relation (\ref{d16}) shows that
\begin{gather}
  \beta(x,y)\beta(xy,z)S_{xyz}=\beta(x,y)S_zS_{xy}=S_zS_yS_x=\nonumber\\
  =\beta(y,z)S_{yz}S_x=\beta(x,yz)\beta(y,z)S_{xyz}.\nonumber
\end{gather}
It remains to compare the first and the last term.
\end{proof}

\begin{lemma}\label{beta1}
  Let $x\in G_2/G_0$. Then
  \begin{equation}\nonumber
  \prod_{y\in G_2/G_0}\beta(x,y)=1,\quad\prod_{y\in G_2/G_0}\beta(y,x)=1.
  \end{equation}
\end{lemma}
\begin{proof}
Exponent $n(x)$ was defined such that $x^{n(x)}=h(x)\in G_0$. Now,
\begin{gather}\nonumber
  S_y\pi(h(x)^{-1})\stackrel{(\ref{d12})}{=}S_yS_x^{n(x)}=\beta(x,y)S_{xy}
  S_x^{n(x)-1}=\beta(x,y)\beta(x,xy)S_{x^2y}S_x^{n(x)-2}=\nonumber\\
  =\ldots=\beta(x,y)\beta(x,xy)\ldots
  \beta(x,x^{n(x)-1}y)S_{x^{n(x)}y}\label{d23}.
\end{gather}
Now, use that
\begin{equation}\label{d24}
  S_{x^{n(x)}y}=S_{h(x)y}=S_y\pi(h(x)^{-1})
\end{equation}
Relations (\ref{d23}) and (\ref{d24}) show that
\begin{equation}\label{d26}
  \beta(x,y)\beta(x,xy)\ldots\beta(x,x^{n(x)-1})=1.
\end{equation}
Since $n(x)$ divides $n=|G_2/G_0|$, $G_2/G_0$ can be written as a disjoint
union of classes $\{x^{n(x)-1}y,\ldots,xy,y\}$. Since the product (\ref{d26})
is equal to 1 for each y in $G_2/G_0$, the first relation follows.
The second formula can be proved similarly.
\end{proof}

\begin{proposition}\label{betaroot}
  Let $x,y\in G_2$. Then
  \begin{equation}\nonumber
    \beta(x,y)^n=1.
  \end{equation}
\end{proposition}
\begin{proof}
Lemma \ref{as} shows that
\begin{equation}\label{d28}
  \beta(x,y)\beta(xy,z)=\beta(x,yz)\beta(y,z).
  \end{equation}
Now, we write (\ref{d28}) for all $z\in G_2/G_0$ and multiply them.
Lemma \ref{beta1} gives
\begin{equation}\nonumber
  \beta(x,y)^n\cdot1=1\cdot1
\end{equation}
and it proves the proposition.
\end{proof}

\noindent{\bfseries Remark.} Relation (\ref{d12}) was needed only in the proof
of Lemma (\ref{beta1}). Once, we have Proposition (\ref{betaroot}), we can
multiply operators $S_x$ by $\xi_n^l$ for any $l\in\Z$. Proposition
(\ref{betaroot}) will be still valid. Sometimes it will lead to a simpler set of
coefficients $\beta(x,y)$.
\vr

\noindent{\bfseries Remark.} A reader familiar with cohomology of groups
will find some similarities between our construction and the second cohomology
group. We say that the group $E$ is the central extension of the group $F$ if
there exists a short exact sequence
\begin{equation}\label{d32}
  1\rightarrow A\stackrel{i}{\rightarrow}E\stackrel{\pi}{\rightarrow}
  F\rightarrow1
\end{equation}
where $A$ is Abelian and central in $E$. Central extensions are parametrized
by $H^2(F,A)$. A key ingredient in that construction are a set-theoretic section
$\sigma:F\rightarrow E$ and a function $f:F\times F\rightarrow A$ such that
\begin{equation}\nonumber
  \sigma(x)\sigma(y)=i(f(x,y))\sigma(x,y).
\end{equation}
For more details see, for example, \cite{br}.
Let us assume that $G=G_1$.
In our construction $F$ from (\ref{d32}) is equal to the quotient $G/G_0$,
$A=Z_0$ and $E=Z_1$ (or $A=G_0$ and $E=Z_1$).
One could say that $\beta(x,y)=\zeta(i(f(x,y)))$ where $\zeta:Z_0\rightarrow\C$
is the character of the given representation. However, in our construction of
$\beta$, intertwining operators $S$ were included. It could be interesting to
consider the case when $G=G_2$ and find relationship between possible
choices for $\beta:G_2/G_0\rightarrow\C$ (satisfying Lemmas \ref{as} and
\ref{beta1}) and extensions of $G_2/G_0$.

\section{Induction}\label{ind}

Let $H$ be the subgroup of $G$ and $(\pi,V)$ the representation of $H$.
The usual approach to the induction procedure is to consider the space
\begin{equation}\nonumber
  \ind_H^GV=\{f:G\rightarrow V\,|\,f(xh)=\pi(h)^{-1}f(x),
  x\in G,h\in H\}
\end{equation}
and the action is given by $(\ind_{G_0}^G\pi(g)(f))(x)=f(g^{-1}x)$.
We will use a different notation and write
\begin{equation}\label{e8}
  W=\ind_H^GV=\left\{w=\sum_{x\in G/H}xv_x\,|\,v_x\in V\right\}
\end{equation}
where sum goes over all classes of $G/H$ and $x$ is some representative.
We can write
\begin{equation}\label{e9}
  W=\{w(v_x)=\sum_{x\in G/H}xv_x\,|\,v_x\in V\}
\end{equation}
or
\begin{equation}\label{e10}
W=\{w=\sum_{x\in G/H}xv_x(w)\,|\,v_x(w)\in V\}.
\end{equation}
Relation (\ref{e9}) emphasizes that any collection of $|G/H|$ vectors from
$V$ determines an element $w$ in $W$.
Relation (\ref{e10}) emphasizes that any $w$ in $W$ determines the
collection of vectors $v_x(w)$.
Usually, we will write elements of $W$ in terms of (\ref{e8}).
The action of $g\in G$ is given by
\begin{equation}\nonumber
  g.\sum_{x\in G/H}xv_x=\sum_{x\in G/H}xh(x)v_x
\end{equation}
for some elements $h(x)\in H$.
The action can be described more precisely if $H$ is the normal subgroup of $G$.
The identity component, $G_ 0$ is the normal subgroup subgroup of $G$.
Also, the subgroup $G_1$ is the normal subgroup of $G$ (Proposition \ref{p1}).

Let $N$ be the normal subgroup of $G$ and $K(N)$ the maximal compact subgroup
of $N$. Since $G/N\cong K/K(N)$, we can write
\begin{equation}\label{e12}
  \sum_{x\in K/K(N)}xv_x
\end{equation}
instead of (\ref{e8}).
The action $\ind_N^G\pi(n)$ for $n\in N$ is given by
\begin{equation}\nonumber
  n.\left(\sum_{x\in G/N}xv_x\right)=\sum_{x\in G/N}x(x^{-1}nx).v_x
\end{equation}
and the action $\ind_{G_0}^G\pi(y)$ for $y\not\in N$ is given by
\begin{equation}\label{e16}
  y.\left(\sum_{x\in G/N}xv_x\right)=\sum_{x\in G/N}yx.v_x.
\end{equation}
The element $yx$ may not be the representative of the class $yxG_0$.
If $z(x,y)$ is the chosen representative of the class $yxG_0$ in (\ref{e8})
then $yx=z(x,y)n(y,x)$ for $n(y,x)\in G_0$ and the "more precise" form of
(\ref{e16}) is
\begin{equation}\nonumber
  y.\left(\sum_{x\in G/N}xv_x\right)=\sum_{x\in G/N}z(x,y)n(y,x).v_x.
\end{equation}
For $(\g,K)$-modules, the action of $X\in\g$ on (\ref{e12}) can be obtained
using (\ref{b4}):
\begin{equation}\nonumber
  X.\left( \sum_{x\in K/K(N)}xv_x\right) =
  \sum_{x\in K/K(N)}x\left(\mathrm{Ad}(k^{-1})X\right).v_x.
\end{equation}

Sometimes, it is convenient to use operators $S$ in (\ref{e8}).
Let us assume that $\pi^x$ is equivalent to $\pi$ for all $x\in G/N$.
Then we write
\begin{equation}\nonumber
  \ind_N^GV=\left\{\sum_{x\in G/N}xS_xv_x\,|\,v\in V\right\}.
\end{equation}
The action of $y\not\in N$ is unchanged but the action of $n\in N$
transforms to
\begin{equation}\nonumber
  n.\left(\sum_{x\in G/N}xS_xv_x\right)=\sum_{x\in G/N}x(x^{-1}nx)S_xv_x
  \stackrel{(\ref{d4})}{=}\sum_{x\in G/N}xS_xnv_x.
\end{equation}
Let us analyze irreducible components of $W=\ind_N^GV$ if $\pi^x$ in that case.
Let us write
\begin{equation}\nonumber
  W=\ind_N^GV=\bigoplus_jW_j.
\end{equation}
In particular, we are interested in the form of elements $w_j\in W_j$

\begin{lemma}\label{dec}
  Let us assume that $N$ is the normal subgroup of $G$,
  $\pi^x$ is equivalent to $\pi$ for all $x\in G/N$
  and $W_j$ is an irreducible component of $\ind_N^GV$.
  Then there exists $m\in\N$,
  and a collection of $\lambda_x^l$, $x\in G/N$, $l\in\{1,\ldots,m\}$
  such that any element $w_j\in W_j$ has the form
  \begin{equation}\nonumber
    w_j=w_j(v_1,\ldots,v_m)=\sum_{x\in G/N}xS_x
    (\lambda_x^1 v_1+\ldots+\lambda_x^mv_m)
  \end{equation}
  where $v_i\in V$, $S_x$ are intertwining operators and $\lambda_x^l\in\C$.
  It should be more correct to write $\lambda_{x,j}^l$, $v_{l,j}$ and $m(j)$
  instead of $\lambda_x^l$, $v_l$ and $m$, but we assume that the index
  $j$ is fixed.
\end{lemma}
\begin{proof}
The subspace $W_j$ contains elements $w_j$ of the form (\ref{e8}). Let us
put elements of $G/N$ in a sequence, starting with $e$. We set $v_1=v_e$.
Now, consider the second element, say $x$.
Let us recall (\ref{e10}) and write $v_x(w_j)$.
If $v_x(w_j)$ is not a function of $v_e(w_j)$ for all $w_j$,
then we choose $v_2(w_j)$ such that $S_xv_2(w_j)=v_x(w_j)$.
If $v_x(w_j)$ depends on $v_e(w_j)$ for all $w_j$, then $v_x(w_j)=A_xv_e(w_j)$
for some linear operator $A_x$ since $W_j$ is a linear space. Now, write
\begin{equation}\label{e24}
  w_j=ev_1(w_j)+xA_xv_1(w_j)+\sum_{y\in(G/N)\setminus\{eN,xN\}}yv_y(w_j).
\end{equation}
Now, we act by $n\in N$ on (\ref{e24}) and get
\begin{equation}\label{e28}
  n.w_j=e\pi(n)v_1(w_j)+x\pi^x(n)A_xv_1(w_j)+
  \sum_{y\in(G/N)\setminus\{eN,xN\}}y\pi^y(n)v_y(w_j).
\end{equation}
If we put $\pi(n)v_1(w_j)$ instead of $v_1(w_j)$ in (\ref{e24}), we get
\begin{equation}\label{e32}
  w_j=e\pi(n)v_1(w_j)+xA_x\pi(n)v_1(w_j)+
  \sum_{y\in(G/N)\setminus\{eN,xN\}}yv_y(w_j).
\end{equation}
One can compare vectors after $x$ in (\ref{e28}) and (\ref{e32}) and get
\begin{equation}\nonumber
  \pi^x(n)A_x=A_x\pi(n).
\end{equation}
It shows that $A_x=\lambda_x^1S_x$.
We continue in this way. In order to simplify notation, we switch to
terminology of (\ref{e8}) again and write $v_z$ instead of $v_z(w_j)$.
Let us assume that we already have independent
elements $v_1,\ldots,v_p$ and come to element $y$.
If $v_y$ does not depend on $v_1,\ldots,v_q$ in the same way for all $w_j$,
then we choose $v_{p+1}$ such that $S_yv_{p+1}=v_y$.
If $v_y$ depends on $v_1,\ldots,v_p$ in the same way for all $w_j$,
then put $v_y=A_y^1v_1+\ldots+A_y^pv_p$ in (\ref{e24}). In order to simplify
notation we will write only element with $y$. The action of $n\in N$ gives
\begin{equation}\label{e34}
  n.w_j=\ldots+y\pi^y(n)(A_y^1v_1+\ldots+A_y^pv_p)+\ldots
\end{equation}
If we put $\pi(n)v_l$ instead of $v_l$ in (\ref{e24}) for $l\in\{1,\ldots,p\}$,
we get
\begin{equation}\label{e36}
  w_j(\pi(n)v_1,\ldots,\pi(n),\ldots)=
  \ldots+y(A_y^1\pi(n)v_1+\ldots+A_y^p\pi(n)v_p)+\ldots
\end{equation}
Since $v_y$ depends on $\pi(n)v_1,\ldots,\pi(n)v_p$ and they coincide in
$n.w_j$ in (\ref{e34}) and in $w_j(\pi(n)v_1,\ldots,\pi(n)v_p,\ldots)$ in
(\ref{e36}), one concludes
\begin{equation}\nonumber
  \pi^y(n)(A_y^1v_1+\ldots+A_y^pv_p)=A_y^1\pi(n)v_1+\ldots+A_y^p\pi(n)v_p
\end{equation}
Hence,
\begin{equation}\nonumber
  \pi^y(n)A_y^l=A_y^l\pi(n)
\end{equation}
for $l\in\{1,\ldots,p\}$. It shows that
\begin{equation}\nonumber
  A_y^l=\lambda_y^lS_y.
\end{equation}
for some $\lambda_y^l\in\C$.
\end{proof}

The reducibility of $\ind_{G_0}^G\pi$ is related to equivalence of $\pi$
and $\pi^x$.

\begin{theorem}
	If $G_2=G_0$ then $\ind_{G_0}^G\pi$ is irreducible.
\end{theorem}
\begin{proof}
The theorem follows from the Frobenius reciprocity theorem. Since
\begin{equation}\nonumber
  \mathrm{Hom}_G\left(\ind_{G_0}^G\pi,\ind_{G_0}^G\pi\right)=
  \mathrm{Hom}_{G_0}\left(\pi,\bigoplus_{x\in G/G_0}x\pi^x\right)
\end{equation}
the left hand side is equal to \C\ if and only if the right hand side is
equal to \C\ and it is valid if and only if $\pi\not\cong\pi^x$.
\end{proof}

It is easy to see directly that $\ind_{G_0}^G\pi$ is reducible if
$\pi\cong\pi^x$. Let $G_x$ be the subgroup of $G$ of the form
$G_0\cup xG_0\cup\ldots\cup x^{n(x)-1}G_0$. In this paper we need the
definition and basic properties of coefficients $\beta$ (we have stopped at
Proposition \ref{betaroot}). Hence, for this comment we assume that
$\beta(x^m,x^k)=1$ for $0\leq m,k\leq n(x)$. Then $\ind_{G_0}^{G_x}\pi$
decomposes in the form $\ind_{G_0}^{G_x}V=\bigoplus_{j=0}^{n(x)-1}U_j$ where
\begin{equation}\nonumber
  U_j=\{u_j(v)=v+\xi_{n(x)}^jxS_xv+\ldots+
  \xi_{n(x)}^{jn(x)-1}x^{n(x)-1}S_x^{n(x)-1}v\,|\,v\in V\}.
\end{equation}

\subsection{$\ind_{G_0}^{G_1}$}\label{ind01}

\begin{theorem}\label{gogi}
  Let $\zeta_\pi:Z_0\rightarrow\C$ be the character of $Z_0$ defined by
  $\pi(z)=\zeta_\pi(z)I_V$ for $z\in Z_0$.
  The representation $\ind_{G_0}^{G_1}\pi$ decomposes in the same way as
  the representation $\ind_{Z_0}^{Z_1}\zeta_\pi$.
\end{theorem}
\begin{proof}
Let us write \C\ for the vector space of $\zeta_\pi$. Then
\begin{equation}\nonumber
  \ind_{Z_0}^{Z_1}\C=\bigoplus_jU_j
\end{equation}
where $U_j$, by Lemma \ref{dec}, have the form
\begin{equation}\label{e40}
  U_j=\left\{u_j(c_1,\dots,c_m)=\sum_{x\in Z_1/Z_0}x
  (\lambda_x^1c_1+\ldots+\lambda_x^mc_m)\,|\,c_l\in\C\right\}
\end{equation}
for some $\lambda_x^l\in\C$. Now,
\begin{equation}\nonumber
  \ind_{G_0}^{G_1}V=\bigoplus_jW_j
\end{equation}
where
\begin{equation}\nonumber
  W_j=\left\{w_j(v_1,\dots,v_m)=\sum_{x\in Z_1/Z_0}
  x(\lambda_x^1v_1+\ldots+\lambda_x^mv_m)\,|\,v_l\in V\right\}.
\end{equation}
for the same $\lambda_x^l$ as in (\ref{e40}).
The action of $g\in G_0$ is given by
\begin{align*}
  g.w_j(v_1,\dots,v_m)=g.\left(\sum_{x\in Z_1/Z_0}
  x(\lambda_x^1v_1+\ldots+\lambda_x^mv_m)\right)=\\
  =\sum_{x\in Z_1/Z_0}x(\lambda_x^1g.v_1+\ldots+\lambda_x^mg.v_m)
  =w_j(g.v_1,\dots,g.v_m).
\end{align*}
Hence, $g.w_j(v_1,\ldots,v_m)\in W_j$.
The action of $x\in Z\setminus Z_0$ on $w_j(v_1,\dots,v_m)$ is the same as the
action of $x$ on $u_j(c_1,\dots,c_m)$.
\end{proof}

\noindent{\bfseries Example.} Let us assume that $Z_1=Q$, the quaternion
group and $Z_0=\Z_2=\{\pm1\}$. It is more convenient to denote
elements of $Q$ by $1$, $x$, $y$ and $xy$ instead of $i$, $j$ and $k$.
Hence $Q=\{\pm1,\pm x,\pm y,\pm xy\}$.
The structure of $G_0$ is not important.
There are two possible cases: $\pi_0|_{\{\pm1\}}$ is the trivial representation
and $\pi_1|_{\{\pm1\}}$ is a nontrivial representation.
The corresponding spaces will be denoted by $V_0$ and $V_1$ and restrictions
to $Z_0$ by $\C_0$ and $\C_1$.

Let us consider the first case. We know that
\begin{equation}\nonumber
  \ind_{Z_0}^{Z_1}\C_0=W_0\oplus W_1\oplus W_2\oplus W_3.
\end{equation}
It is easy to reconstruct spaces $W_j$:
\begin{gather}
  W_0=\left\{w_0(c)=c+xc+yc+xyc\,|\,c\in\C\right\}\nonumber\\
  W_1=\left\{w_1(c)=c+xc-yc-xyc\,|\,c\in\C\right\}\nonumber\\
  W_2=\left\{w_2(c)=c-xc+yc-xyc\,|\,c\in\C\right\}\nonumber\\
  W_3=\left\{w_3(c)=c-xc-yc+xyc\,|\,c\in\C\right\}.\nonumber
\end{gather}
The action of $\pm1$ is trivial and the action of other elements is easy to
reconstruct. For example,
\begin{equation}\nonumber
  y.w_1(c)=y.(c+xc-yc-xyc)=yc+(-xy)c-(-1)c-(-x)c=-w_1(c).
\end{equation}
Now,
\begin{equation}\nonumber
  \ind_{G_0}^{G_1}V_0=U_0\oplus U_1\oplus U_2\oplus U_3
\end{equation}
where
\begin{gather}
  U_0=\left\{u_0(v)=v+xv+yv+xyv\,|\,v\in V\right\}\nonumber\\
  U_1=\left\{u_1(v)=v+xv-yv-xyv\,|\,v\in V\right\}\nonumber\\
  U_2=\left\{u_2(v)=v-xv+yv-xyv\,|\,v\in V\right\}\nonumber\\
  U_3=\left\{u_3(v)=v-xv-yv+xyv\,|\,v\in V\right\}.\nonumber
\end{gather}
The action of $g\in G_0$ is given by
\begin{equation}\nonumber
  g.v_j(v)=v_j(g.v)
\end{equation}
and the action of other elements is the same as above. For example,
\begin{equation}\nonumber
  y.u_1(v)=y.(v+xv-yv-xyv)=yv+(-xy)v-(-1)v-(-x)v=-u_1(v).
\end{equation}

The nontrivial case is more interesting. We know that
\begin{equation}\nonumber
  \ind_{Z_0}^{Z_1}\C_1=W_0\oplus W_1
\end{equation}
and $W_0$ and $W_1$ are equivalent. Now,
\begin{gather}
  W_0=\left\{w_0(c,d)=c+ixc+yd-ixyd\,|\,c,d\in\C\right\}\nonumber\\
  W_1=\left\{w_1(c,d)=c-ixc+yd+ixyd\,|\,c,d\in\C\right\}.\nonumber
\end{gather}
The action of elements of $G$ on $W_j$ is given by
\begin{gather}
  (-1).w_0(c,d)=w_0(-c,-d),\qquad(-1).w_1(c,d)=w_1(-c,-d)\nonumber\\
  x.w_0(c,d)=w_0(-ic,id),\qquad x.w_1(c,d)=w_1(ic,-id)\nonumber\\
  y.w_0(c,d)=w_0(-d,c),\qquad y.w_1(c,d)=w_1(-d,c).\nonumber
\end{gather}
Then
\begin{equation}\nonumber
  \ind_{G_0}^{G_1}V_0=U_0\oplus U_1
\end{equation}
where
\begin{gather}
  U_0=\left\{u_0(v,t)=v+ixv+yt-ixyt\,|\,v,t\in V\right\}\label{e44}\\
  U_1=\left\{u_1(v,t)=v-ixv+yt+ixyt\,|\,v,t\in V\right\}\label{e48}
\end{gather}
and $U_0$ and $U_1$ are equivalent and the intertwining operator sends
$u_0(v,t)$ to $u_1(v,t)$.
The action of elements of $G$ on $W_j,\;(j=0,1)$ is given by
\begin{gather}
  g.u_j(v,t)=u_j(g.v,g.t)\nonumber\\
  x.u_j(v,t)=u_j(-i(-1)^kv,i(-1)^kt)\nonumber\\
  y.u_j(v,t)=u_j(-t,v).\nonumber
\end{gather}

\subsection{$\ind_{G_0}^{G_2}$}\label{ind02}

We prefer to induce from $G_0$ in order to avoid some technical difficulties.
This step is similar to previous one.
However, intertwining operators $S_x$ have to be used.

Let $n=|G_2/G_0|$, $\ds \xi_n=\cos\frac{2\pi}n+i\sin\frac{2\pi}n$ and
$\Z_n=\{\xi_n^j\,|\,j\in\{0,\dots,n-1\}\}$.
Proposition \ref{betaroot} shows that $\beta(x,y)=\xi_n^j$ for
$k\in\{0,\ldots,n-1\}$.
Let $B$ be the subgroup of $\Z_n$ generated by all $\beta(x,y)$.
Then $B$ has the form
\begin{equation}\nonumber
  B=\Z_m=\left\{\xi_m^j=\left(\cos\frac{2\pi}{m}+i\sin\frac{2\pi}{m}\right)\,|\,
  j\in\{0,\ldots,m-1\}\right\}
\end{equation}
for some $m$ which divides $n$.
It indicates that we should consider the group $F_2$ for which the
underlying set has the form $(G_2/G_0)\times B$ and the group operation is
given by
\begin{equation}\label{e52}
  (x,\alpha)(y,\beta)=(xy,\beta(x,y)\alpha\beta)
\end{equation}
It is easy to check the group operation is well defined. For example,
associativity for $F$ follows from associativity for $G$ (or $G_2$).
Associativity can be also proved formally using Proposition \ref{as}.
Let $F_0=F_2\cap G_0=\{e\}\times B=\{(e,\xi^k)\,|\,k\in\{0,\dots,m-1\}\}$.

We define the representation $\eta_m$ of $F_0$ on \C\ by
\begin{equation}\nonumber
  \eta_m((e,\xi^j)).z=\xi^jz.
\end{equation}

\begin{theorem}\label{gogii}
  The representation $\ind_{G_0}^{G_2}\pi$ decomposes in the same way as
  the representation $\ind_{F_0}^{F_2}\eta_m$.
\end{theorem}
\begin{proof}
The proof is similar to the proof of Theorem \ref{gogi}.
However, we have to emphasize that this time, we have an abstract group
$F_2$ instead of subgroup $Z_1$. Actually, we could use an abstract group
$F_1$ in Theorem \ref{gogi} but the subgroup $Z_1$ is more natural choice.
Let
\begin{equation}\nonumber
  \ind_{F_0}^{F_2}\C=\bigoplus_jU_j
\end{equation}
where $U_j$ have the form
\begin{equation}\label{e56}
  U_j=\left\{u_j(c_1,\dots,c_m)=\sum_{x\in F_2/F_0}
  x(\lambda_x^1c_1+\ldots+\lambda_x^mc_m)\,|\,c_j\in\C\right\}.
\end{equation}
Now,
\begin{equation}\nonumber
  \ind_{G_0}^{G_2}V=\bigoplus_jW_j
\end{equation}
where
\begin{equation}\nonumber
  W_j=\left\{w_j(v_1,\dots,v_m)=\sum_{x\in F_2/F_0}
  xS_x(\lambda_x^1v_1+\ldots+\lambda_x^mv_m)\,|\,v_l\in V\right\}
\end{equation}
for the same $\lambda_x^l$ as in (\ref{e56}).
The action of $g\in G_0$ is given by
\begin{gather*}
  g.w_j(v_1,\dots,v_m)=g.\left(
  \sum_{x\in F_2/F_0}xS_x(\lambda_x^1v_1+\ldots+\lambda_x^mv_m)\right)
  \stackrel{(\ref{d4})}{=}\\ \stackrel{(\ref{d4})}{=}\sum_{x\in F_2/F_0}
  xS_x(\lambda_x^1g.v_1+\ldots+\lambda_x^mg.v_m)=w_j(g.v_1,\dots,g.v_m).
\end{gather*}
Hence, $g.w_j(v_1,\ldots,v_m)\in W_j$.
The action of $y\in Z\setminus Z_0$ on $w_j(v_1,\dots,v_m)$ is given by
\begin{gather*}
  y.w_j(v_1,\dots,v_m)=y.\left(\sum_{y\in F_2/F_0}
  xS_x(\lambda_x^1v_1+\ldots+\lambda_x^mv_m)\right)=\\
  =\sum_{x\in F_2/F_0}
  yxS_xS_yS_y^{-1}(\lambda_x^1v_1+\ldots+\lambda_x^mv_m)=\\
  =\sum_{x\in F_2/F_0}
  yxS_xS_y(\lambda_x^1S_y^{-1}v_1,\ldots,\lambda_x^mS_y^{-1}v_m).
\end{gather*}
It shows that the action of $y\in Z\setminus Z_0$ on $w_j(v_1,\dots,v_m)$ is
the same as the action of $y$ on $u_j(c_1,\dots,c_m)$ (up to adding $S_y^{-1}$).
Coefficients $\beta$ appear in the action of $y$ on $w\in\ind_{G_0}^{G_2}\pi$
and in the action of $y$ on $u\in\ind_{F_0}^{F_2}\eta_m$.
\end{proof}

\noindent{\bfseries Example.} We will consider the example which is very
similar to the previous one.
Let $G_0=(SL(2,\R)\times SL(2,\R))/\{I\times I,(-I)\times(-I)\}$.
This time, the action will be written in terms of $(\g,K)$-modules.

For the beginning, let us consider only one one $SL(2,\R)$.
We use the terminology and results from \cite{vo}. The complexified Lie algebra
$\g=\mathfrak{sl}(2,\C)$ is spanned by
\begin{equation}\nonumber
  H=-i\left[\begin{array}{cc}0&1\\-1&0\end{array}\right],\quad
  X=\frac12\left(\left[\begin{array}{cc}1&0\\0&-1\end{array}\right]+
  i\left[\begin{array}{cc}0&1\\1&0\end{array}\right]\right),
\end{equation}
\begin{equation}\nonumber
  Y=\frac12\left(\left[\begin{array}{cc}1&0\\0&-1\end{array}\right]-
  i\left[\begin{array}{cc}0&1\\1&0\end{array}\right]\right).
\end{equation}
It acts on the vector space $V$ spanned by vectors $v_j$ for all $j\in\Z$.
The action of $\mathfrak{sl}(2,\C)$ on $V$ is given by
\begin{gather}
  H.v_j=jv_j\nonumber\\
  X.v_j=\frac12\left(\lambda+(j+1)\right)v_{j+2}\nonumber\\
  Y.v_j=\frac12\left(\lambda-(j-1)\right)v_{j-2}.\nonumber
\end{gather}
for some $\lambda\in\C$. Our $V$ is irreducible if and only if
$\lambda+j+1\neq0,\;\forall j$. Now, we introduce the element $x$ of the form
\begin{eqnarray}\label{e60}
  x=\left[\begin{array}{cc}1&0\\0&-1\end{array}\right]
\end{eqnarray}
and it is easy to check that
\begin{gather}
  Ad(x)H=-H\nonumber\\
  Ad(x)X=Y\nonumber\\
  Ad(x)Y=X.\nonumber
\end{gather}
Also, it is easy to check that $\pi^x$ is isomorphic to $\pi$ and the intertwining
operator $S_x$ is given by
\begin{equation}\label{e64}
  S_xv_j=v_{-j}.
\end{equation}
The $K$-type $V_k$ is spanned by vector $v_k$.
It shows that $S_xV_j=V_{-j}$. Hence $\ind_{G_0}^G$
should act on $V_j\oplus V_{-j}$. It is easy to show that
\begin{equation}\nonumber
  \ind_{G_0}^G(V_j\oplus V_{-j})=Z_j^0\oplus Z_j^1
\end{equation}
for $j>0$ and
\begin{equation}\nonumber
  \ind_{G_0}^GV_0=Z_0^0\oplus Z_0^1
\end{equation}
where
\begin{gather}
  Z_j^0=\mathrm{span}(v_j+xS_xv_j,v_{-j}+xS_xv_{-j})=
    \mathrm{span}(v_j+xv_{-j},v_{-j}+xv_j)\nonumber\\
  Z_j^1=\mathrm{span}(v_j-xS_xv_j,v_{-j}-xS_xv_{-j})=
    \mathrm{span}(v_j-xv_{-j},v_{-j}-xv_j)\nonumber\\
  Z_0^0=\mathrm{span}(v_0+xS_xv_0)=\mathrm{span}(v_0+xv_0)\nonumber\\
  Z_0^1=\mathrm{span}(v_0-xS_xv_0)=\mathrm{span}(v_0-xv_0).\nonumber
\end{gather}
Now, $(\g,K)$-module $\ind_{G_0}^GV$ is reducible and
\begin{equation}\nonumber
  \ind_{G_0}^GV=U_0+U_1.
\end{equation}
If $K$-types of $V$ are odd, then
\begin{equation}\nonumber
  U_0|_K=\bigoplus_{j\;\mathrm{odd}}Z_j^0\qquad\mathrm{and}\qquad
  U_1|_K=\bigoplus_{j\;\mathrm{odd}}Z_j^1.
\end{equation}
If $K$-types of $G_0$ are even, then
\begin{equation}\nonumber
  U_0|_K=Z_0^0\oplus\bigoplus_{j\;\mathrm{even},\;j>0}U_j^0
  \qquad\mathrm{and}\qquad
  U_1|_K=Z_0^1\oplus\bigoplus_{j\;\mathrm{even},\;j>0}U_j^1.
\end{equation}
For example,
\begin{gather}
  x.(w_j+xw_{-j})=w_{-j}+xw_j\qquad x.(w_{-j}+xw_j)=w_j+xw_{-j}\nonumber\\
  x.(w_j-xw_{-j})=-(w_{-j}-xw_j)\qquad x.(w_{-j}-xw_j)=-(w_j+xw_{-j})\nonumber\\
  X.(w_{-2}+xw_2)=\frac12(\lambda-1)(w_0+xw_0)\nonumber\\
  X.(w_{-2}-xw_2)=\frac12(\lambda-1)(w_0-xw_0)\nonumber\\
  Y.(w_2+xw_{-2})=\frac12(\lambda-1)(w_0+xw_0)\nonumber\\
  X.(w_j+xw_{-j})=\frac12(\lambda+j+1)(w_{j+2}+xw_{-j-2}).\nonumber
\end{gather}

Now, let us consider our group $G_0$ acting on the space
\begin{equation*}
  \mathrm{span}\{v_j^1\otimes v_l^2\,|\,j,l\in\Z\}.
\end{equation*}
The group $G$ has 4 connected components.
We add elements similar to (\ref{e60}) (by abuse of notation we use $x$ again):
$x=i\cdot\mathrm{diag}(1,-1,1,1)$, $y=j\cdot\mathrm{diag}(1,1,1,-1)$ and
$z=k\cdot\mathrm{diag}(1,-1,1,-1)$. Multiplication by $i$, $j$ and $k$ does not
mean multiplication by all entries. We simply want to obtain $xy=-yx$.
Thanks to (\ref{e64}), it is easy to find intertwining operators:
\begin{equation}\nonumber
  S_xv_j^1\otimes v_l^2=v_{-j}^1\otimes v_l^2,\quad
  S_yv_j^1\otimes v_l^2=v_j^1\otimes v_{-l}^2,\quad
  S_zv_j^1\otimes v_l^2=v_{-j}^1\otimes v_{-l}^2.
\end{equation}
Actually, we should multiply operators $S$ by $i$ in order to satisfy
(\ref{d12}). However, Remark after Proposition \ref{betaroot} shows that
we can omit multiplication by $i$.
The next step is to find coefficients $\beta(x,y)$.
Since intertwining operators $S_x$ commute, $\beta(x,y)=1$, $\beta(y,z)=1$,
$\beta(z,x)=1$, $\beta(y,x)=-1$, $\beta(z,y)=-1$ and $\beta(x,z)=-1$.
It shows that $B=\{\pm1\}$ and $F_2=\Z_2\times\Z_2\times B$ as a set and the
group operation (\ref{e52}) shows that $F_2\cong Q$, where $Q$ is the group
of quaternions. Also, $F_0=\Z_2$.
We are in situation of the previous Example. Using (\ref{e44}) and (\ref{e48}), one
can easily get
\begin{gather}
  U_0=\left\{u_0(v,t)=v+ixv+yt-ixyt\,|\,v=v^1\otimes v^2,t=t^1\otimes t^2
    \in V\right\}\nonumber\\
  U_1=\left\{u_1(v,t)=v-ixv+yt+ixyt\,|\,v=v^1\otimes v^2,t=t^1\otimes t^2
    \in V\right\}\nonumber.
\end{gather}

\subsection{$\ind_{G_2}^G$}\label{ind03}

The following result is known. For example, see \cite{du}. However, we
put the theorem in order to complete the induction procedure. Also, our
proof follows ideas developed in this paper.

\begin{theorem}
The representation $\ind_{G_2}^G\rho$ is irreducible for the irreducible
representation $(\rho,W)$ of $G_2$.
\end{theorem}
\begin{proof}
Let $(\pi,V)$ be the irreducible representation of $G_0$ and
let $(\rho,W)$ be one irreducible component of $(\ind_{G_0}^{G_2}\pi,
\ind_{G_0}^{G_2}V)$. Now, let us consider $\ind_{G_2}^GW$ and let $U$ be
one irreducible component. We want to show that $U=\ind_{G_2}^GW$.
Let us mention that $G_2$ do not have to be a normal subgroup of $G$.
By (\ref{e8}) elements of $U$ can be written in the form
\begin{equation}\nonumber
  u=\sum_{y\in G/G_2}yw_y
\end{equation}
where $w_y\in W$.
If there is no relations among $w_y$'s, then $U=\ind_{G_2}^GW$.
Hence, we have to assume that at least one $w$, say $w_p$ is a function of
other $w$'s. Since $U$ is a linear space, we can assume that
\begin{equation}\label{e68}
  w_p=\sum_{y\in(G/G_2)\setminus\{pG_2\}}A_yw_y
\end{equation}
for some linear operators $A_y$ and write
\begin{equation}\label{e72}
  u=\sum_{y\in (G/G_2)\setminus\{pG_2\}}yw_y+
  p\sum_{y\in(G/G_2)\setminus\{pG_2\}}A_yw_y.
\end{equation}
Lemma \ref{dec} shows that (\ref{e72}) can be written in the form
\begin{gather}
  u=\sum_{y\in (G/G_2)\setminus\{pG_2\}}y\sum_{x\in G_2/G_0}xS_x
  \left(\lambda_{y,x}^1 v_{y,x}^1+\ldots+
  \lambda_{y,x}^{m(y,x)}v_{y,x}^{m(y,x)}\right)+\nonumber\\
  +p\sum_{x\in G_2/G_0}xS_x\sum_j A_x^jv_x^j\label{e76}
\end{gather}
where $v_x^j\in\{v_{y,x}^l\}$ and $A_x^j\in GL(V)$.
Now, let us choose one $v_{q,x}^l$ to be an arbitrary element of $V$, denote
it by $v$ and set all other $v_{y,x}^l=0$. Then (\ref{e76}) transforms to
\begin{equation}\label{e80}
  u=u(v)=q\sum_{x\in G_2/G_0}xS_x\lambda_xv+p\sum_{x\in G_2/G_0}xS_xA_xv.
\end{equation}
Elements $u(v)$ form a $G_0$-invariant linear subspace
of $U$. The action of $g\in G_0$ on (\ref{e80}) produces
\begin{equation}\label{e84}
  \pi(g)u(v)=q\sum_{x\in G_2/G_0}xS_x\lambda_x\pi^q(g)v+
  p\sum_{x\in G_2/G_0}xS_x\pi^p(g)A_xv.
\end{equation}
We can write $\pi^q(g)v$ instead of $v$ in (\ref{e80}) and get
\begin{equation}\label{e88}
  u(\pi^q(g))=q\sum_{x\in G_2/G_0}xS_x\lambda_x\pi^q(g)v+
  p\sum_{x\in G_2/G_0}xS_xA_x\pi^q(g)v.
\end{equation}
One can compare (\ref{e84}) and (\ref{e88}) and obtain
\begin{equation}\nonumber
  \pi^p(g)A_xv=A_x\pi^q(g)v
\end{equation}
or
\begin{equation}\nonumber
  \pi^{q^{-1}p}(g)A_xv=A_x\pi(g)v.
\end{equation}
It gives a contradiction since $q^{-1}p\notin G_2$. Hence, the assumption
(\ref{e68}) is not valid. It shows that $U=\ind_{G_2}^GW$.
\end{proof}

\section*{Acknowledgement}

This work was supported by the QuantiXLie Centre of Excellence, a project co financed by the Croatian Government and European Union through the European Regional Development Fund - the Competitiveness and Cohesion Operational Programme (Grant KK.01.1.1.01.0004).

\end{document}